\documentclass[11pt,a4paper]{article}

\usepackage[utf8]{inputenc}
\usepackage{graphicx}   
\usepackage{float}   
\usepackage{amsmath, amsthm, amssymb, amsfonts}
\usepackage{mathrsfs}
\usepackage{geometry}
\usepackage{parskip} 

\usepackage[numbers,sort&compress]{natbib}

\usepackage{hyperref}
\hypersetup{
    colorlinks=true,
    linkcolor=blue,
    filecolor=magenta,      
    urlcolor=cyan,
    citecolor=red,
}

\newtheorem{theorem}{Theorem}[section]
\newtheorem{lemma}[theorem]{Lemma}

\newtheorem{definition}[theorem]{Definition}
\theoremstyle{remark}
\newtheorem{remark}[theorem]{Remark}
\theoremstyle{definition}

\title{\textbf{Existence of Positive Mild Eigenfunctions for Caputo Fractional Semilinear Evolution Equations with Nonlocal Initial Conditions}}

\author{
    Sajid Ullah$^{1,*}$, Assia Guezane-Lakoud$^{2}$
}
\date{} 
\begin{document}
\maketitle

\noindent
$^{1}$ Department of Mathematics and Computer Science, University of Calabria,\\ 
Ponte P. Bucci 30B, Rende (CS), Italy.\\
\texttt{sajid.ullah@unical.it}

\medskip

\noindent
$^{2}$ Laboratory of Systems and Advanced Materials, Department of Mathematics,\\
Badji Mokhtar-Annaba University, Annaba, Algeria.\\
\texttt{a\_guezane@yahoo.fr}
\medskip

\noindent
$^{*}$\,Corresponding author.
\begin{abstract}
We study the existence of positive eigenpairs for a class of Caputo fractional autonomous evolution equations with nonlocal initial condition within the framework of Banach lattices. The autonomous linear operator generates a compact strongly continuous semigroup of contractions, while the nonlinearity is a Caratheodory map. The mild eigenfunction is represented via the compact Mittag--Leffler operator families, we work within a positive cone of continuous functions and establish a uniform lower bound for the solution operator on the boundary. We apply the Birkhoff--Kellogg type theorem in cone for the existence of eigenpair. Our approach requires neither Lipschitz continuity of the nonlinearity nor the compactness of nonlocal initial operator, allowing for broad applicability to periodic, multi-point, and integral-type initial conditions. The theoretical results are applied to a parabolic fractional partial differential equation.
\end{abstract}

\smallskip
\section{Introduction}
In this paper, we address the existence of a positive eigenvalue $\lambda > 0$ and a nonnegative mild eigenfunction associated to the abstract Cauchy problem
\begin{equation}\label{eq:main}
  \begin{cases}
    {}^C D_t^{\beta} u(t) = \mathscr{A}u(t) + \lambda\, \mathscr{F}(t, u(t)), & t \in [0,1], \\[4pt]
    u(0) = \lambda\, \mathscr{H}[u],
  \end{cases}
\end{equation}
where $0 < \beta < 1$, the operator $\mathscr{A}$ generates a compact and positive $C_0$-semigroup of contractions on a Banach lattice $V$, $\mathscr{F}$ is a Carath\'eodory nonlinearity mapping into the positive cone of $V$ and $\mathscr{H} : C([0,1];V) \to V$ is a nonlocal functional. ${}^C D_t^{\beta}$ denotes the Caputo fractional derivative of order $\beta$, and defined by
\[
{}^{C}D_t^\beta u(t)
= \frac{1}{\Gamma(1-\beta)} \int_0^t (t - s)^{-\beta} u'(s)\, ds.
\]
In this paper, we deal with the abstract theory of fractional evolution equations in the setting of Banach spaces, the nonlocal initial conditions, and the positivity of eigenpair for semilinear operator equations in ordered Banach spaces.

Fractional-order differential operators have been developed as essential tools for modeling processes where memory effects, generalized diffusion and hereditary properties play a fundamental role. The ability to encode nonlocal temporal correlations makes them well suited for viscoelastic constitutive relations, fractional parabolic dynamics, subdiffusive transport in heterogeneous media; we will refer the reader to~\cite{Kilbas2006, Zhou2016} for an extensive history of fractional calculus, including both its theory and applications. The study of autonomous evolution equations in infinite dimensional Banach spaces has developed rapidly in the past few decades. 

 The concept of a $\beta$-mild solution was introduced by Wang and Zhou~\cite{Wang2011}, for the Caputo fractional Cauchy problem via probability density functions and strongly continuous semigroups, proving the existence and uniqueness by a singular Gronwall-type inequality together with the Leray--Schauder fixed point theorem for compact maps. Zhou, Zhang, and Shen~\cite{Zhou2013} applied the framework to the Riemann--Liouville setting and relaxed the compactness requirement on the semigroup by means of the Hausdorff measure of noncompactness. The main feature of their approaches is the way in which the mild solution is represented by the Mittag-Leffler families of operators $\mathcal{S}_\beta(t), \mathcal{T}_\beta(t)$, which replace the classical semigroup in the variation of parameters formula. The compactness and positivity properties of these families play a very important role in the analysis of this work. We refer to ~\cite{Chang2009, Jaradat2008, Mophou2009a, Mophou2009b} and the references therein for further work in this direction.

Nonlocal initial conditions of the form $u(0) = \mathscr{H}[u]$, introduced by Byszewski~\cite{Byszewski1991, Byszewski1998} as a generalization 
of the classical datum $u(0) = u_0$, arise naturally when the system state must be inferred from weighted averages, multipoint 
measurements, or integrals over the time interval—as in Deng's~\cite{Deng1993} model of gas diffusion in a transparent tube.
The subsequent theory has developed along several directions. Liang, Liu, and Xiao~\cite{Liang2004} identified the 
support of the nonlocal operator as a key structural parameter, exploiting compact norm-continuity of the semigroup 
for $t>0$ when that support is bounded away from zero; Cardinali, Precup, and Rubbioni~\cite{Cardinali2015} gave a unified 
treatment tracing the progressive transition from Volterra to Fredholm type as the support grows. Hernández, dos Santos, and Azevedo~\cite{Hernandez2011} extended the analysis to the fractional-power scale $X_\alpha = \mathcal{D}((\mathscr{A})^{\alpha})$, 
handling nonlocal terms involving spatial derivatives. An approximation solvability method—combining Yosida approximations, 
continuation principles, and weak-topology arguments—has since removed compactness requirements entirely: see Benedetti, Loi, and Taddei~\cite{Benedetti2017} for the reflexive Banach-space setting, Xu, Colao, and Muglia~\cite{ColaoMuglia2021} for unbounded 
intervals, and Benedetti and Ciani~\cite{Benedetti2022} for superlinear growth. In the fractional framework, Chen and Feng~\cite{Chen2024} combined this approximation technique with a Hartman-type inequality and the Leray--Schauder continuation principle to 
establish existence for equation~\eqref{eq:main} (with $\lambda=1$) under superlinear $\mathscr{F}$, without Lipschitz regularity of $\mathscr{F}$ or compactness of $\mathscr{H}$.

A distinct direction concerns \emph{positive} solutions in ordered Banach spaces, relevant when the unknown represents a nonnegative physical quantity. In the integer-order setting, Infante and Maciejewski~\cite{InfanteMaciejewski} studied parabolic systems with nonlinear nonlocal initial conditions, proving existence, localization, and multiplicity of positive solutions via the Granas fixed point index in suitable cones. Infante and Rubbioni~\cite{InfanteRubbioni} investigated abstract semilinear evolution equations with functional initial conditions using a Birkhoff--Kellogg type theorem of Krasnosel'skiĭ and Ladyženskiĭ in ordered Banach spaces with normal cones, establishing the existence of a positive eigenvalue $\lambda > 0$ and a corresponding nonnegative mild eigenfunction of prescribed norm. Their framework handles periodic, multipoint, and integral-average conditions in a unified way and has been applied to reaction-diffusion models from heat conduction. 

Despite the advances described above, these threads have remained largely separate in the fractional setting: existence results for fractional evolution equations~\cite{Chen2024} focus on sign-indefinite solutions without addressing positivity or spectral structure, while the eigenvalue theory of~\cite{InfanteRubbioni} is confined to first-order equations and does not extend directly to the Caputo setting, where the Mittag-Leffler families $\mathcal{S}_\beta(t)$ and $\mathcal{T}_\beta(t)$ require independent analysis. The present paper fills this gap by developing a Birkhoff--Kellogg type eigenvalue theory for~\eqref{eq:main} in Banach lattices, replacing the Leray--Schauder principle of~\cite{Chen2024} with a cone-based approach following~\cite{InfanteRubbioni} and extending it to the fractional framework. Under verifiable lower-bound assumptions on the Carathéodory nonlinearity $\mathscr{F}$ and the nonlocal functional $\mathscr{H}$—requiring neither Lipschitz continuity of $\mathscr{F}$ nor compactness of $\mathscr{H}$—we show that for each prescribed norm $\alpha > 0$ there exist a positive eigenvalue $\lambda_\alpha > 0$ and a nonnegative mild eigenfunction $u_\alpha$ with $\|u_\alpha\|_\infty = \alpha$. The framework subsumes periodic, multipoint, and integral-type nonlocal conditions as special cases, and simultaneously generalizes the integer-order theory of~\cite{InfanteRubbioni, InfanteMaciejewski} and the fractional existence theory of~\cite{Wang2011, Chen2024}.
 
The rest of the paper is organized as follows. The preliminary material, such as the key definitions and properties of the Caputo derivative, the Mittag-Leffler operator families $\mathcal{S}_\beta(t)$ and $\mathcal{T}_\beta(t)$ and their
positivity and compactness and the cone theory applicable to ordered Banach spaces, can be found in Section \ref{sec:prelim}. Section \ref{main} contains the formulation of the problem, the main hypotheses and the existence theorem. In Section~\ref{App} we show how the abstract theory can be applied to practical examples.

\section{Preliminaries}\label{sec:prelim}
We recall some basic definitions and results that will be used throughout the paper. Let $V$ be a real vector space and $Z$ be a nonempty subset of $V$. We call $Z$ a convex cone if 
\begin{itemize}
    \item[Z1)] $Z+Z\subset Z$;
    \item[Z2)] $\zeta Z \subset Z$ for all $\zeta\ge 0$;\\
In addition, $Z$ is called a pointed convex cone if 
    \item[Z3)] $Z\cap (-Z)=\{0\}$. 
\end{itemize}
A pointed convex cone induces a partial order $\preceq$ on $V$ by 
\[
    u \preceq v \;\text{iff}\; v-u \in Z.
\]
The partial order is compatible with the algebraic structure of the space, meaning that the following properties hold:
\begin{itemize}
    \item[(1)] $u\preceq v\; \implies \; u+w\preceq v+w$, for every $u, v, w \in V$ 
    \item[(2)] $u\preceq v\; \implies \; \zeta u\preceq \zeta v$, for every $u, v \in V$ and $\zeta \ge 0$.
\end{itemize}
An ordered vector space $(V, \preceq)$ is a real vector space equipped with such a partial order relation. The set 
\[
    V^+=\{ u\in V\;|\; 0\preceq u\}  
\]
is called a positive cone of $V$, which is also a pointed convex cone.  A vector lattice $(V, \preceq)$ is an order vector space such that every pair of elements has an infimum and a supremum. 

We represent by $(\Lambda, \Sigma, \mu)$ a measure space and by $(V, \, \|\cdot \|_V)$ a real Banach space. A normed vector lattice $(V,\,\|\cdot \|,\,\preceq)$ is a vector lattice equipped with the norm satisfying $|u|\preceq |v|\implies\|u\|\le \|v\|$. A Banach lattice $(V, \, \|\cdot \|_V, \preceq)$ is a complete normed vector lattice. Let $(\Lambda, \Sigma, \mu)$ be a measure space and let $g: \Lambda\to V$ be a strongly measurable function. Suppose that there exists a sequence of simple functions $\{\theta_n\}_{n\in \mathbb{N}}$ in $V$ such that $\lim_{n\to \infty}\|g(x)-\theta_n(x)\|_V=0$ for $\mu$-almost all $x\in \Lambda$. A strongly $\mu$-measurable function $g: \Lambda\to V$ is Bochner integrable if there exists a sequence $\{\theta_n\}_{n\in \mathbb{N}}$ of $V$-step functions such that the real measurable function $\|g-\theta_n\|_V\in L^1(\Lambda)$ and $\lim_{n\to \infty}\int \|g-\theta_n\|_Vd\mu=0$. Then for every $I\in \Sigma$, the Bochner integral of $g$ over $I$ is defined by
    \[
        \int_I g\,d\mu= \lim_{n\to \infty} \int_I \theta_n \,d\mu.
    \]

Moreover, the Bochner integral is monotone in the sense of partial ordering $\preceq$. For further properties of the Bochner integral, we refer the reader to \cite{AB2006}.

Let $(V,\,\|\cdot\|)$ be a Banach space, we represent by $\mathcal{L}(V)$ the space of all bounded linear operators on $V$. In this paper, we denote a strongly continuous semigroup by $C_0$-semigroup.

\begin{definition}\cite{Pazzy}
Let $V$ be a Banach space. A one-parameter family of bounded linear operators $\{T(t)\}_{0 \leq t}$  from $V$ into $V$ is called a \emph{semigroup} on $V$ if
\begin{itemize}
    \item[($S_1$)] $T(0) = I$, where $I$ is the identity operator on $V$.
    \item[($S_2$)] $T(t + s) = T(t)T(s)$ for every $t, s \geq 0$ (the semigroup property).
    \item[($S_3$)] for each fixed $u \in V$, the mapping $t \mapsto T(t)u$ is continuous on $t\ge 0$.
\end{itemize}
\end{definition}

The infinitesimal generator of the semigroup $T(t)$ is a  linear operator $\mathscr{A}: \mathcal{D}(\mathscr{A}) \rightarrow V$ defined by
\[
\mathcal{D}(\mathscr{A}) = \left\{ u \in X : \lim_{t \downarrow 0}\frac{T(t)u - u}{t} \text{ exists} \right\}
\]
and
\[
\mathscr{A}u = \lim_{t \downarrow 0}\frac{T(t)u - u}{t} \quad \text{for } u \in \mathcal{D}(\mathscr{A}).
\]
Let $T(t)$ be a $C_0$-semigroup, then there exist constants $\delta$ and $M$ such that 
\[
    \|T(t)\|\le M e^{t\delta} \quad for \quad t\ge 0, 
\]
a  $C_0$-semigroup is called contraction $C_0$-semigroup if $\delta=0$ and $M=1$ (see Section 1.2 \cite{Pazzy}). 
\begin{definition} \cite{Pazzy}
    A $C_0$ semigroup $T(t)$ is called compact for $t > t_0$ if for every $t>t_0$, $T(t)$ is a compact operator. $T(t)$ is called compact if it is compact for $t> 0$.
\end{definition}
\begin{theorem}(Them. 3.3\cite{Pazzy})\label{Them. 3.3}
    A $C_0$ semigroup $T(t)$ is a compact semigroup iff $T(t)$ is continuous in the uniform operator topology for $t>0$ and the resolvent operator is compact. 
\end{theorem}   
    Let \( u \in V \). We define two operators \( \mathcal{S}_\beta(t) \) and \( \mathcal{T}_\beta(t) \) for \( t \ge 0 \) by
\[
\mathcal{S}_\beta(t)u = \int_0^\infty k_\beta(\tau)\, T\!\big(t^\beta \tau\big)\, u \, d\tau,
\]
\[
\mathcal{T}_\beta(t)u = \beta \int_0^\infty \tau\, k_\beta(\tau)\, T\!\big(t^\beta \tau\big)\, u \, d\tau,
\]
where
\[
k_\beta(\tau) = \frac{1}{\pi \beta} \sum_{n=1}^\infty (-\tau)^{n-1} \frac{\Gamma(n\beta + 1)}{n!} \sin(n\pi \beta), 
\quad \tau \ge 0.
\]

Moreover, \( k_\beta \) satisfies
\[
k_\beta(\tau) \ge 0, \quad \tau \ge 0, 
\qquad 
\int_0^\infty k_\beta(\tau)\, d\tau = 1,
\]
and
\[
\int_0^\infty \tau\, k_\beta(\tau)\, d\tau = \frac{1}{\Gamma(1+\beta)}.
\]
The following properties of the operators \( \mathcal{S}_\beta(t) \) and \( \mathcal{T}_\beta(t) \) are given in ~\cite{Wang2011,Zhou2013,CZL2020}, and they will be used in our analysis.

\begin{lemma}\label{L1}
Assume that $\{T(t)\}_{t \geq 0}$ is a compact $C_0$-semigroup of contractions generated by the linear operator $\mathscr{A}$, then the following hold:

\begin{enumerate}
    \item $\mathcal{S}_\beta(t)$ and $\mathcal{T}_\beta(t)$ are linear and bounded operators, i.e. for any $u \in V$, we have
    \[
        \|\mathcal{S}_\beta(t)u\|_V \leq \|u\|_V, \quad
        \|\mathcal{T}_\beta(t)u\|_V \leq \frac{\beta}{\Gamma(\beta+1)}\|u\|_V;
    \]

    \item the operators $\mathcal{S}_\beta(t)$ and $\mathcal{T}_\beta(t)$ are strongly continuous, i.e. for every $u \in V$ and $0 \leq t_2 < t_1 \leq T$, we have
    \[
        \|\mathcal{S}_\beta(t_2)\,u - \mathcal{S}_\beta(t_1)\,u\|_V \to 0, \quad
        \|\mathcal{T}_\beta(t_2)\,u - \mathcal{T}_\beta(t_1)\,u\|_V \to 0
        \quad \text{as } t_2 \to t_1;
    \]

    \item $\mathcal{S}_\beta(t)$ and $\mathcal{T}_\beta(t)$ are compact operators for every $t > 0$.
\end{enumerate}
\end{lemma}

Let $(V, \succeq, \|\cdot\|)$ be a Banach lattice, we recall the definition of positive semigroup.  A $C_0$ semigroup $\{T(t)\}_{t\ge 0}$ on a Banach lattice $V$ is called positive if each operator $T(t)$ is positive, i.e.,
    \[
        T(t)u\succeq 0 \; \text{for all}\; t\ge 0, \;\; \text{whenever} \; u\succeq 0.
    \]
If $T(t)$ is positive, then both $\mathcal{S}_\beta(t)$ and $\mathcal{T}_\beta(t)$ are also positive.


\section{Main Result}\label{main}
In this section, we provide an existence result for the mild eigenfunction and the corresponding eigenvalue. first, we recall some useful ingredients. Let $(V, \succeq, \|\cdot\|)$ is Banach lattice, where the partial order $\succeq$ is induced by a given normal cone $Z \subset V$. A cone $Z$ is said to be normal if there exists $c>0$ such that 
\[
   0 \preceq u \preceq v \;\Rightarrow\; \|u\| \leq c \|v\|, \quad \text{for all } u, v \in Z,
\]
where $\preceq$ is the partial ordering induced by $Z$.
Define $K$ as the positive cone in $C([0,1];V)$ by
\[
    K=\{u \in C([0,1]; V) : u(t) \succcurlyeq 0, \text{ for every } t \in [0,1]\}.
\]
Then $C([0,1];V)$, equipped with the partial order induced by $K$, is itself a Banach lattice. For every $\alpha\in (0,\; \infty)$, we consider the following
\[
K_\alpha := \{u \in K : \|u\|_{\infty} < \alpha\}, \quad \overline{K_\alpha} := \{u \in K : \|u\|_{\infty} \leq \alpha\}, \quad \partial K_\alpha := \{u \in K : \|u\|_{\infty} = \alpha\}.
\]
and
\[
Z_\alpha := \{u \in Z : \|u\|_V < \alpha\}, \quad \overline{Z_\alpha} := \{u \in Z : \|u\|_V \leq \alpha\}, \quad \partial Z_\alpha := \{u \in Z : \|u\|_V = \alpha\}.
\]
We use the following Birkhoff-Kellogg type theorem on cones, for the existence of eigenvalue and corresponding mild eigenfunction.
\begin{theorem}\cite[Theorem~5.5]{Krasnoselskii1964}, \cite{KrasnoselskiiLadyzhenskii1954}\label{B-K}
    Let $(V,\| \, \|)$ be a real Banach space, let $\Phi:\overline{K}_{\alpha}\to K$ be compact and suppose that 
    \[
        \inf_{x\in \partial K_{\alpha}}\|\Phi x\|>0.
    \]
    Then there exist $\lambda_{0}\in (0,+\infty)$ and $u_{0}\in \partial K_{\alpha}$ such that $u_{0}=\lambda_{0} \Phi u_{0}.$
\end{theorem} 
\begin{definition}
    We say that a couple $(\lambda, u)$ is a solution of the problem \eqref{eq:main}, where $\lambda\in \mathbb{R}^+$ and $u\in C([0, 1], V)$ if the following integral equation
    \begin{equation}
        u(t)=\lambda\mathcal{S}_{\beta}(t)\mathscr{H}[u]+\lambda\int_0^t(t-r)^{\beta-1}\mathcal{T}_{\beta}(t-r)\mathscr{F}(r, u(r))dr,\quad\text{for every}\; t\in[0, 1].
    \end{equation}
    holds. In this case, we call $u$ a mild eigenfunction of \eqref{eq:main} corresponding to the eigenvalue $\lambda$.
\end{definition}
Now, we can state and prove our main result as follows.
\begin{theorem}\label{existence}
    Assume that the following assumptions hold.
    \begin{enumerate}
        \item[(a1)] The infinitesimal generator $\mathscr{A}$ generates a compact and positive $C_0$ semigroup of contractions.
        \item[(f1)] The nonlinear map $\mathscr{F}: [0,\,1]\times \bar{Z}_{\alpha}\to Z$ is a Caratheodory map, i.e., $\mathscr{F}(\cdot, \, u):[0,\,1]\to Z$ is measurable for every $u\in \bar{Z}_{\alpha}$ and $\mathscr{F}(t, \,\cdot): \bar{Z}_{\alpha}\to Z$ is continuous for every $t\in [0, \,1]$.
        \item[(f2)] $\mathscr{F}: [0,\,1]\times \bar{Z}_{\alpha}\to Z$ is mapping bounded sets into bounded sets and there exists $\gamma_{\alpha}:[0,\,1]\to Z$ such that 
        \[
            \mathscr{F}(t,\,u(t))\succeq \gamma_{\alpha}(t) \;\;\text{for every}\; t\in [0,\,1],\;u\in \partial K_{\alpha}. 
        \]
        \item[(h1)] $\mathscr{H}:C([0,\,1]\times \bar{Z}_{\alpha})\to Z$ is a bounded linear operator, mapping bounded sets into bounded sets, and there exists $\Phi_{\alpha}\in  Z$ such that 
        \[
            \mathscr{H}[u]\succeq \Phi_{\alpha}\;\;\text{for every}\; u\in \partial K_{\alpha}
        \]
        \item[(h2)] There exists $t_0\in(0,\,1]$ such that
        \[
            \big\|\mathcal{S}_{\beta}(t_0)\Phi_{\alpha}+\int_0^{t_0}(t_0-r)^{\beta-1}\mathcal{T}_{\beta}(t_0-r)\gamma_{\alpha}(r)dr\big\|>0,
        \]
        where $\mathcal{S}_{\beta}(t)$ and $\mathcal{T}_{\beta}(t)$ satisfy conditions $(1)-(3)$ of Lemma \ref{L1}.
    \end{enumerate}
    Then there exists a positive eigenvalue $\lambda_{\alpha}$ and the corresponding eigenfunction $u_\alpha\in\partial K_{\alpha}$ such that $(\lambda_{\alpha},\, u_{\alpha})$ solve the problem \eqref{eq:main}.
\begin{proof}
    As we assumed $\mathscr{A}$ generates a compact and positive $C_0$ semigroup of contractions, then we have 
    \[
        \|T(t)\|_{\mathcal{L}(V)}\le 1\quad \text{for all}\; t\ge 0.
    \]
Now, we consider the operator $\mathscr{T}:\bar{K}_{\alpha}\to K$ defined by 
\[
    \mathscr{T}u(t)=\mathcal{S}_{\beta}(t)\mathscr{H}[u]+\int_0^{t}(t-r)^{\beta-1}\mathcal{T}_{\beta}(t-r)\mathscr{F}(r,\,u(r))dr\quad t \in [0,1].
\]
We proceed to verify that $\mathscr{T}$ is cone invariant and is a compact operator. We then establish the lower bound needed to apply the 
Birkhoff--Kellogg theorem.

Let $u \in \bar{K}_\alpha$ be arbitrary. Since $u(t) \succeq 0$ for every $t\in[0,1]$ and by $(h1)$, we have $\mathscr{H}[u] \succeq 0$. Similarly, $\mathscr{F}(r, u(r)) \succeq 0$ for every $r\in[0,1]$ by $(f1)$ and $(f2)$.

Since $\mathscr{A}$ generates a positive $C_0$-semigroup by (a1), both $\mathcal{S}_\beta(t)$ and $\mathcal{T}_\beta(t)$ are positive operators. Therefore,
\[
    \mathcal{S}_\beta(t)\mathscr{H}[u] \succeq 0 \quad \text{and}\quad \mathcal{T}_\beta(t-r)\,\mathscr{F}(r, u(r)) \succeq 0.
\]
By the monotonicity of the Bochner integral with respect to the partial order $\succeq$, and the fact that $(t-r)^{\beta-1} \geq 0$ for $0 \leq r \leq t$, we conclude that $(\mathscr{T}u)(t) \succeq 0$ for every $t\in[0,1]$.

Moreover, $\mathscr{T}u$ is continuous on $[0,1]$ by the strong continuity of $\mathcal{S}_\beta(\cdot)$ and $\mathcal{T}_\beta(\cdot)$ (Lemma~\ref{L1}, part (2)) and the properties of the Bochner integral. Thus $\mathscr{T}u \in C([0,1];V)$, and so $\mathscr{T}u \in K$.

Let $\{u_n\}_{n\in\mathbb{N}} \subset \bar{K}_\alpha$ be a sequence converging to 
$u^* \in \bar{K}_\alpha$ in the sup-norm $\|\cdot\|_\infty$. For every $t\in[0,1]$ 
and $n\in\mathbb{N}$, using Lemma~\ref{L1} part (1):
\begin{align*}
    &\|(\mathscr{T}u_n)(t) - (\mathscr{T}u^*)(t)\|\le \|\mathcal{S}_\beta(t)\|_{\mathcal{L}(V)}\|\mathscr{H}[u_n] - \mathscr{H}[u^*]\| \\
    &\quad + \int_0^t (t-r)^{\beta-1}\|\mathcal{T}_\beta(t-r)\|_{\mathcal{L}(V)}
        \|\mathscr{F}(r, u_n(r)) - \mathscr{F}(r, u^*(r))\|\,dr \\
    &\le \|\mathscr{H}[u_n] - \mathscr{H}[u^*]\|
      + \frac{\beta}{\Gamma(\beta+1)}\int_0^1(t-r)^{\beta-1}
        \|\mathscr{F}(r,u_n(r))-\mathscr{F}(r,u^*(r))\|\,dr.
\end{align*}

By assumption \textbf{(f2)}, the set $\mathscr{F}([0,1]\times\bar{Z}_\alpha)$ is bounded, so 
there exists $M_\alpha > 0$ such that
\begin{equation}\label{eq:2}
     \|\mathscr{F}(r, v)\| \le M_\alpha \quad \text{for every } r\in[0,1],\; v\in\bar{Z}_\alpha.
\end{equation}
Therefore
\[
    \|\mathscr{F}(r, u_n(r)) - \mathscr{F}(r, u^*(r))\| \le 2M_\alpha
    \quad\text{for every } r\in[0,1],\; n\in\mathbb{N}.
\]
Since the function $(t-r)^{\beta-1}$ is integrable on $[0,t]$ for $\beta\in(0,1)$, 
we may apply the Lebesgue Dominated Convergence Theorem. As $u_n \to u^*$ uniformly, 
the continuity of $\mathscr{F}(r, \cdot)$ gives $\mathscr{F}(r, u_n(r)) \to \mathscr{F}(r, u^*(r))$ for every $r$. Similarly, the continuity of $\mathscr{H}$ gives $\mathscr{H}[u_n]\to \mathscr{H}[u^*]$.

Hence
\[
    \lim_{n\to\infty}\|\mathscr{T}u_n - \mathscr{T}u^*\|_\infty = 0,
\]
and $\mathscr{T}$ is continuous.

To prove $\mathscr{T}(\bar{K}_\alpha)$ is relatively compact, we decompose $\mathscr{T} = \mathcal{H} + \mathcal{G}$, where
\[
    (\mathcal{H}u)(t) = \mathcal{S}_\beta(t)\mathscr{H}[u], \qquad
    (\mathcal{G}u)(t) = \int_0^t (t-r)^{\beta-1}\mathcal{T}_\beta(t-r)\,\mathscr{F}(r,u(r))\,dr.
\]

The set $\bar{K}_\alpha$ is bounded, so by \textbf{(h1)} the set $\mathcal{H}[\bar{K}_\alpha]$ 
is bounded in $V$. For every fixed $t\in[0,1]$, the operator $\mathcal{S}_\beta(t)$ is 
compact (Lemma~\ref{L1} part (3)) and bounded. Therefore the set
\[
    \mathcal{H}(\bar{K}_\alpha)(t) = \bigl\{\mathcal{S}_\beta(t)\mathscr{H}[u] : u\in\bar{K}_\alpha\bigr\}
\]
is relatively compact in $V$ for every $t\in[0,1]$.

Now we need to show, for every $t\in [0, 1]$ the relative compactness of 
\[
    \mathcal{G}(\bar{K}_{\alpha})(t)=\Big\{\int_0^t (t-r)^{\beta-1}\mathcal{T}_\beta(t-r)\,\mathscr{F}(r,u(r))\,dr\;:\; u\in \bar{K}_{\alpha}\Big\}.
\]
By assumption $(f2)$, there exists $M_{\alpha}>0$ such that
$\|\mathscr{F}(r,v)\|\le M_{\alpha}$ for every $r\in[0,1]$ and $v\in\bar{Z}_{\alpha}$.
Using Lemma~\ref{L1}, for every $u\in\bar{K}_{\alpha}$ and
$t\in[0,1]$,
\[
    \|(\mathcal{G}u)(t)\|
    \;\le\;
    \int_{0}^{t}(t-r)^{\beta-1}\,
    \|\mathcal{T}_{\beta}(t-r)\|\,\|\mathscr{F}(r,u(r))\|\,dr
    \;\le\;
    \frac{\beta M_{\alpha}}{\Gamma(\beta+1)}
    \int_{0}^{t}(t-r)^{\beta-1}\,dr
    \;=\;
    \frac{M_{\alpha}}{\Gamma(\beta+1)}.
\]
Hence, the family $\mathcal{G}(\bar{K}_{\alpha})$ is uniformly bounded in $C([0,1];V)$.

Let $t \in (0, 1]$, $\eta > 0$ with $t - \eta \geq 0$. For arbitrary small $\varepsilon > 0$, arbitrary $\bar{\mu} > 0$ and every 
$u \in\bar{K}_{\alpha} $, by using definition of $\mathcal{T}_{\beta}$ we decompose $\mathscr{G}(u)(t) $ as
\begin{align*}
\mathscr{G}(u)(t) 
&= \beta \int_0^{t-\eta} \int_{\bar{\mu}}^{\infty} 
    (t-s)^{\beta-1} \tau k_\beta(\tau)\, T\!\left((t-s)^\beta \tau\right) 
    \mathscr{F}(s, u(s))\, d\tau\, ds \\
&\quad + \beta \int_0^{t-\eta} \int_0^{\bar{\mu}} 
    (t-s)^{\beta-1} \tau k_\beta(\tau)\, T\!\left((t-s)^\beta \tau\right) 
    \mathscr{F}(s, u(s)\, d\tau\, ds \\
&\quad + \int_{t-\eta}^{t} (t-s)^{\beta-1} \mathcal{T}_\beta(t-s)\, \mathscr{F}(s, u(s))\, ds.
\end{align*}
Factoring out $T(\eta^\beta \bar{\mu})$ from the first integral by using $(S_2)$, we rewrite 
this as
\begin{align*}
\mathscr{G}(u)(t)
&= T\!\left(\eta^\beta \bar{\mu}\right) 
    \left[ \beta \int_0^{t-\eta} \int_{\bar{\mu}}^{\infty} 
    (t-s)^{\beta-1} \tau k_\beta(\tau)\, 
    T\!\left((t-s)^\beta \tau - \eta^\beta \bar{\mu}\right) 
    \mathscr{F}(s, u(s))\, d\tau\, ds \right] \\
&\quad + \beta \int_0^{t-\eta} \int_0^{\bar{\mu}} 
    (t-s)^{\beta-1} \tau k_\beta(\tau)\, T\!\left((t-s)^\beta \tau\right) 
    \mathscr{F}(s, u(s))\, d\tau\, ds \\
&\quad + \int_{t-\eta}^{t} (t-s)^{\beta-1} \mathcal{T}_\beta(t-s)\, \mathscr{F}(s, u(s))\, ds.
\end{align*}
Define the operator 
$\mathcal{F}_{\eta, \bar{\mu}} : \mathcal{G}(\bar{K}_{\alpha})(t) \to K$ by
\begin{align*}
\mathcal{F}_{\eta,\bar{\mu}}\!\left(\mathcal{G}(u)(t)\right)
&=T\!\left(\eta^\beta \bar{\mu}\right) 
    \left[ \beta \int_0^{t-\eta} \int_{\bar{\mu}}^{\infty}
    (t-s)^{\beta-1} \tau k_\beta(\tau)\,
    T\!\left((t-s)^\beta \tau - \eta^\beta \bar{\mu}\right)
    \mathscr{F}(s, u(s))\, d\tau\, ds \right].
\end{align*}
Since $T(\eta^\beta \bar{\mu})$ is a compact operator, and 
$\mathcal{G}$ maps bounded subsets into bounded 
subsets of $C([0,1]; V)$, implies $\mathcal{F}_{\eta,\bar{\mu}}(\bar{K}_{\alpha})(t)$ is relatively compact.

Moreover, for every $u\in \bar{K}_{\alpha}$, we estimate the remainder:
\begin{align*}
&\left\| \mathcal{F}_{\eta,\bar{\mu}}\!\left(\mathcal{G}(u)(t)\right) 
- \mathcal{G}(u)(t) \right\|_V \\
&\leq \beta \int_0^{t-\eta} \int_0^{\bar{\mu}} 
    (t-s)^{\beta-1} \tau k_\beta(\tau)\, 
    \left\| T\!\left((t-s)^\beta \tau\right) \mathscr{F}(s, u(s)) \right\|_V d\tau\, ds \\
&+\int_{t-\eta}^{t} (t-s)^{\beta-1} 
    \left\| \mathcal{T}_\beta(t-s)\, \mathscr{F}(s, u(s)) \right\|_V ds \\
&\leq \beta M_{\alpha}
    \left( \int_0^{t-\eta} \left( \int_0^{\bar{\mu}} \tau k_\beta(\tau)\, 
     d\tau \right) ds \right) + \frac{\beta M_{\alpha}}{\Gamma(\beta+1)} \int_{t-\eta}^{t} (t-s)^{\beta-1} ds.
\end{align*}
By the uniform boundedness of $\mathscr{F}$ and the contractive property of $T$,  we conclude that
\[
\lim_{\eta,\, \bar{\mu} \to 0} 
\left\| \mathcal{F}_{\eta,\bar{\mu}}\!\left(\mathcal{G}(u)(t)\right) 
- \mathcal{G}(u)(t) \right\|_V = 0,
\]
uniformly in $u \in \bar{K}_{\alpha}$. Therefore, $\mathcal{G}(\bar{K}_{\alpha})(t)$ is relatively compact for each $t \in (0, 1]$.
Combining the two parts, $\mathscr{T}(\bar{K}_\alpha)(t)$ is relatively compact 
for every $t\in[0,1]$.

\textit{Equicontinuity of $\mathscr{T}(\bar{K}_\alpha)$.}
Let $0 \le t_1 < t_2 \le 1$ and $u\in\bar{K}_\alpha$. Using Lemma~\ref{L1} and 
the bound $M_\alpha$ on $\mathscr{F}$, together with $N_\alpha > 0$ such that $\|\mathscr{H}[u]\|\le N_\alpha$ 
for all $u\in\bar{K}_\alpha$ (from \textbf{(h1)}):
\begin{align*}
    &\|(\mathscr{T}u)(t_2) - (\mathscr{T}u)(t_1)\| \\
    &\le \|[\mathcal{S}_\beta(t_2) - \mathcal{S}_\beta(t_1)]\mathscr{H}[u]\|
    + \left\|\int_{t_1}^{t_2}(t_2-r)^{\beta-1}\mathcal{T}_\beta(t_2-r)\mathscr{F}(r,u(r))\,dr\right\| \\
    &\quad + \left\|\int_0^{t_1}\bigl[(t_2-r)^{\beta-1}\mathcal{T}_\beta(t_2-r)
      - (t_1-r)^{\beta-1}\mathcal{T}_\beta(t_1-r)\bigr]\mathscr{F}(r,u(r))\,dr\right\| \\
    &\le \|[\mathcal{S}_\beta(t_2) - \mathcal{S}_\beta(t_1)]\mathscr{H}[u]\|
    + \frac{\beta M_\alpha}{\Gamma(\beta+1)}\cdot\frac{(t_2-t_1)^{\beta}}{\beta} \\
    &\quad + M_\alpha\int_0^{t_1}
      \left|(t_2-r)^{\beta-1} - (t_1-r)^{\beta-1}\right|
      \|\mathcal{T}_\beta(t_2-r)\|_{\mathcal{L}(V)}\,dr \\
    &\quad + \left\|\int_0^{t_1}(t_1-r)^{\beta-1}\bigl[\mathcal{T}_\beta(t_2-r)
      - \mathcal{T}_\beta(t_1-r)\bigr]\mathscr{F}(r,u(r))\,dr\right\|.
\end{align*}

By strong continuity of $\mathcal{S}_\beta(\cdot)$(Lemma~\ref{L1} part (2)), $\|[\mathcal{S}_\beta(t_2) - \mathcal{S}_\beta(t_1)]\mathscr{H}[u]\|\to 0$ as $t_2\to t_1$ and $\frac{\beta M_\alpha}{\Gamma(\beta+1)}\cdot\frac{(t_2-t_1)^{\beta}}{\beta} \to 0$ as $t_2\to t_1$. The third term $ M_\alpha\int_0^{t_1}\left|(t_2-r)^{\beta-1} - (t_1-r)^{\beta-1}\right|\|\mathcal{T}_\beta(t_2-r)\|_{\mathcal{L}(V)}\,dr$ tends to $0$ as $t_2\to t_1$, by the dominated convergence theorem. 

We will rewrite the final term for $\eta>0$ with $t_1-\eta>0$ as 
\begin{align*}
    &\left\|\int_0^{t_1}(t_1-r)^{\beta-1}\bigl[\mathcal{T}_\beta(t_2-r)
      - \mathcal{T}_\beta(t_1-r)\bigr]\mathscr{F}(r,u(r))\,dr\right\|\\[5pt]
      &\le \int_0^{t_1-\eta}\|(t_1-r)^{\beta-1}\bigl[\mathcal{T}_\beta(t_2-r)
      - \mathcal{T}_\beta(t_1-r)\bigr]\mathscr{F}(r,u(r))\|\,dr\\
      &\quad+\int_{t_1-\eta}^{t_1}(t_1-r)^{\beta-1}\|\mathcal{T}_\beta(t_2-r)\mathscr{F}(r,u(r))\|\,dr
      +\int_{t_1-\eta}^{t_1} (t_1-r)^{\beta-1}\|\mathcal{T}_\beta(t_1-r)\mathscr{F}(r,u(r))\|\,dr
\end{align*}
By the definition of $\mathcal{T}_{\beta}$, we have

\begin{align*}
   &\left\|\int_0^{t_1}(t_1-r)^{\beta-1}\bigl[\mathcal{T}_\beta(t_2-r)
      - \mathcal{T}_\beta(t_1-r)\bigr]\mathscr{F}(r,u(r))\,dr\right\|\\[5pt]
      &\le  \beta\int_0^{t_1-\eta}\int_{\bar\mu}^{\infty}
   \tau\,k_\beta(\tau)\,
   \bigl\|(t_1-r)^{\beta-1}\bigl[T\!\bigl((t_2-r)^\beta\tau\bigr)
          -T\!\bigl((t_1-r)^\beta\tau\bigr)\bigr]
          \mathscr{F}(r,u(r))\bigr\|\,d\tau\,dr\\
&\quad
+\beta\int_0^{t_1-\eta}\int_0^{\bar\mu}
   \tau\,k_\beta(\tau)\,
   \bigl\|\bigl[(t_1-r)^{\beta-1}\bigl[T\!\bigl((t_2-r)^\beta\tau\bigr)
          -T\!\bigl((t_1-r)^\beta\tau\bigr)\bigr]
          \mathscr{F}(r,u(r))\bigr\|\,d\tau\,dr\\
&\quad
+\int_{t_1-\eta}^{t_1}(t_1-r)^{\beta-1}
   \bigl\|\mathcal{T}_\beta(t_2-r)\mathscr{F}(r,u(r))\bigr\|\,dr
+\int_{t_1-\eta}^{t_1}(t_1-r)^{\beta-1}
   \bigl\|\mathcal{T}_\beta(t_1-r)\mathscr{F}(r,u(r))\bigr\|\,dr\\[5pt]
   \end{align*}
\begin{align*}
&\le \beta\,
\int_0^{t_1-\eta}\int_{\bar\mu}^{\infty}
   \tau\,k_\beta(\tau)\,(t_1-r)^{\beta-1}\\
&\qquad\times
   \bigl\|T\!\bigl((t_2-r)^\beta\tau-\eta^\beta\bar\mu\bigr)
         -T\!\bigl((t_1-r)^\beta\tau-\eta^\beta\bar\mu\bigr)
   \bigr\|_{\mathcal{L}(V)}
   \bigl\|T(\eta^\beta\bar\mu)\bigr\|_{\mathcal{L}(V)}
   \bigl\|\mathscr{F}(r,u(r))\bigr\|\,d\tau\,dr\\
&\quad  +\beta\int_0^{t_1-\eta}\int_0^{\bar\mu}
   \tau\,k_\beta(\tau)\,(t_1-r)^{\beta-1}\\
&\qquad\times
   \bigl\|T\!\bigl((t_2-r)^\beta\tau-\eta^\beta\tau\bigr)
         -T\!\bigl((t_1-r)^\beta\tau-\eta^\beta\tau\bigr)
   \bigr\|_{\mathcal{L}(V)}
   \bigl\|T(\eta^\beta\tau)\bigr\|_{\mathcal{L}(V)}
   \bigl\|\mathscr{F}(r,u(r))\bigr\|\,d\tau\,dr\\
&\quad
+\frac{\beta M_\alpha}{\Gamma(\beta+1)}
   \int_{t_1-\eta}^{t_1}(t_1-r)^{\beta-1}\,dr
+\frac{\beta M_\alpha}{\Gamma(\beta+1)}
   \int_{t_1-\eta}^{t_1}(t_1-r)^{\beta-1}\,dr
   \end{align*}
\begin{align*}
&\le \beta M_{\alpha}\,
\int_0^{t_1-\eta}\int_{\bar\mu}^{\infty}
   \tau\,k_\beta(\tau)\,(t_1-r)^{\beta-1}\\
&\qquad\times
   \bigl\|T\!\bigl((t_2-r)^\beta\tau-\eta^\beta\bar\mu\bigr)
         -T\!\bigl((t_1-r)^\beta\tau-\eta^\beta\bar\mu\bigr)
   \bigr\|_{\mathcal{L}(V)}\,d\tau\,dr\\
&\quad  +\beta M_{\alpha}\,\int_0^{t_1-\eta}\int_0^{\bar\mu}
   \tau\,k_\beta(\tau)\,(t_1-r)^{\beta-1}\\
&\qquad\times
   \bigl\|T\!\bigl((t_2-r)^\beta\tau-\eta^\beta\tau\bigr)
         -T\!\bigl((t_1-r)^\beta\tau-\eta^\beta\tau\bigr)
   \bigr\|_{\mathcal{L}(V)}\,d\tau\,dr+\frac{2\eta \beta M_\alpha}{\Gamma(\beta+1)} 
\end{align*}
The first and second terms tend to $0$ by dominated convergence theorem and continuity of $T$ in the uniform operator norm topology Theorem~\ref{Them. 3.3}, as $t_2\to t_1$ and the last term tends to $0$ as $\eta\to 0$. All bounds are independent of $u\in\bar{K}_\alpha$, so the family 
$\mathscr{T}(\bar{K}_\alpha)$ is equicontinuous.

Since $\mathscr{T}(\bar{K}_\alpha)$ is equicontinuous and relatively compact at each 
$t\in[0,1]$. Hence, Arzel\`{a}--Ascoli theorem implies that $\mathscr{T}$ is a compact operator.

It remains to show that
\[
    \inf_{u\in\partial K_\alpha}\|\mathscr{T}u\|_\infty > 0.
\]
Let $u\in\partial K_\alpha$ be arbitrary, and let $t\in[0,1]$. Since $\mathscr{H}[u]\succeq\Phi_\alpha$ 
by assumption $(h1)$, and $\mathcal{S}_\beta(t)$ is a positive linear operator, we have
\[
    \mathcal{S}_\beta(t)\mathscr{H}[u] \succeq \mathcal{S}_\beta(t)\Phi_\alpha.
\]
Similarly, since $\mathscr{F}(r, u(r))\succeq\gamma_\alpha(r)$ for every $r\in[0,t]$ 
by assumption $(f2)$, and $\mathcal{T}_\beta(t-r)$ is positive and 
$(t-r)^{\beta-1}\ge 0$, the Bochner integral is monotone:
\[
    \int_0^t(t-r)^{\beta-1}\mathcal{T}_\beta(t-r)\mathscr{F}(r,u(r))\,dr
    \succeq \int_0^t(t-r)^{\beta-1}\mathcal{T}_\beta(t-r)\gamma_\alpha(r)\,dr.
\]
Adding these two inequalities,
\[
    (\mathscr{T}u)(t) \succeq \mathcal{S}_\beta(t)\Phi_\alpha 
    + \int_0^t (t-r)^{\beta-1}\mathcal{T}_\beta(t-r)\gamma_\alpha(r)\,dr 
    \quad \text{for every } t\in[0,1].
\]
Since the normal cone $Z\subset V$ is normal, there exists a constant $c>0$ such that 
$0\preceq x\preceq y$ implies $\|x\|\le c\|y\|$. Applying this to the above,
\[
    \|(\mathscr{T}u)(t)\| \ge \frac{1}{c}
    \left\|\mathcal{S}_\beta(t)\Phi_\alpha 
    + \int_0^t(t-r)^{\beta-1}\mathcal{T}_\beta(t-r)\gamma_\alpha(r)\,dr\right\|.
\]
Evaluating at $t = t_0$ (from assumption \textbf{(h2)}),
\[
    \|\mathscr{T}u\|_\infty \ge \|(\mathscr{T}u)(t_0)\| \ge 
    \frac{1}{c}\left\|\mathcal{S}_\beta(t_0)\Phi_\alpha 
    + \int_0^{t_0}(t_0-r)^{\beta-1}\mathcal{T}_\beta(t_0-r)\gamma_\alpha(r)\,dr\right\|
    > 0.
\]
The right-hand side is strictly positive by \textbf{(h2)} and is independent of 
$u\in\partial K_\alpha$. Therefore,
\[
    \inf_{u\in\partial K_\alpha}\|\mathscr{T}u\|_\infty > 0.
\]
We have verified that $\mathscr{T}: \bar{K}_\alpha \to K$ is a compact operator 
satisfying
\[
    \inf_{u\in\partial K_\alpha}\|\mathscr{T}u\|_\infty > 0.
\]
By the Birkhoff--Kellogg type theorem (Theorem~\ref{B-K}), there exist 
$\lambda_\alpha\in(0,+\infty)$ and $u_\alpha\in\partial K_\alpha$ such that
\[
    u_\alpha = \lambda_\alpha \mathscr{T}(u_\alpha).
\]
Unfolding the definition of $\mathscr{T}$, this means
\[
    u_\alpha(t) = \lambda_\alpha\mathcal{S}_\beta(t)\mathscr{H}[u_\alpha]
    + \lambda_\alpha\int_0^t(t-r)^{\beta-1}\mathcal{T}_\beta(t-r)
      \mathscr{F}(r, u_\alpha(r))\,dr, \quad \forall\, t\in[0,1].
\]
That is, $(\lambda_\alpha, u_\alpha)$ is a positive eigenvalue-eigenfunction pair 
solving the problem ~\eqref{eq:main}, with $u_\alpha\in\partial K_\alpha\subset K$ a nonnegative mild eigenfunction.
\end{proof}
\end{theorem}

\section{Applications}\label{App}

In this section, we apply the theoratical result to a concrete example of fractional parabolic partial differential equations. Throughout this Section, we let $\Omega\subset\mathbb{R}^{k}$ be a
bounded domain with $C^{2}$-boundary $\partial\Omega$, and set the Banach lattice $V=L^{p}(\Omega)$ for some fixed $2\le p<\infty$, equipped with $L^P$ norm. We will work in the positive cone,
\[
    Z=\bigl\{u\in L^{p}(\Omega):u(x)\ge 0\;\text{a.e.}\;x\in\Omega\bigr\},
\] 
and 
\[
    K=\{u\in C([0, 1],\,L^{p}(\Omega))\;:\; u(t,x)\ge \sigma\|u\|_{\infty}e_1(x)\;\text{a.e., for all}\;t\in [0, 1]\},
\]
where $\sigma\in (0, 1)$ is fixed $e_1\in Z\setminus \{0\}$ and is the first positive (normalized) eigenfunction of $-\Delta$ under Dirichlet boundary conditions, i.e.
\[
     -\Delta e_{1}=\mu_{1}\,e_{1}\text{ in }\Omega,\qquad
    e_{1}\big|_{\partial\Omega}=0,\qquad
    e_{1}(x)>0\text{ in }\Omega.
\]
Which is a normal cone with normality constant $c=1$. For every $v\in\partial K_{\alpha} $
\begin{equation}\label{cone-bound}
    u(t,x)\;\ge\;\sigma_{0}\,\alpha\,e_{1}(x)
    \quad\text{a.e., for all }t\in[0,1].
\end{equation} We define the linear operator
$\mathscr{A}:D(\mathscr{A})\subset V\to V$ by $\mathscr{A}u=-\Delta u$, with domain $D(\mathscr{A})=W^{2,p}(\Omega)\cap W^{1,p}_{0}(\Omega)$(see, e.g., \cite{Pazzy} Sec:8.3). It is classical that $-A$ generates a compact and positive $C_0$-semigroup of contractions $\{T(t)\}_{t\ge 0}$ on $L^{p}(\Omega)$; thus
assumption $(a1)$ is satisfied. For every fixed $\alpha>0$ and $\phi\in Z$, the operator families $\mathcal{S}_\beta(t)$ and $\mathcal{T}_\beta(t)$
are positive, bounded, and compact for $t>0$ by Lemma~\ref{L1}.

\subsection*{Example 1}

Consider the following problem:
\begin{equation}\label{eq:ex1}
\begin{cases}
{}^{C}D_{t}^{\beta}\,u(t,x)-\Delta u(t,x)
    =\lambda\,\mathscr{F}_{1}(t,x,u(t,x)),
    &(t,x)\in[0,1]\times\Omega,\\[4pt]
\displaystyle u(0,x)=\lambda\int_{0}^{1}\omega(s)\,u(s,x)\,ds,
    &x\in\Omega,\\[4pt]
u(t,x)=0,
    &(t,x)\in[0,1]\times\partial\Omega,
\end{cases}
\end{equation}
where $0<\beta<1$, $\omega\in L^{1}(0,1)$ with $\omega(s)\ge\omega_{0}>0$
a.e.\ on $[0,1]$ for some constant $\omega_{0}>0$,, and the
nonlinear reaction term is chosen as
\[
    \mathscr{F}_{1}(t,x,u(t, x))=\rho(t)\,\phi_{0}(x)+\sigma(t)\,\frac{u(t, x)}{1+u(t, x)},
    \qquad (t,x,v)\in[0,1]\times\Omega \quad u(t, x)\ge 0,
\]

where $\phi_{0}\in Z\setminus\{0\}$ is a fixed nonnegative function in $L^{p}(\Omega)$,
and $\rho,\sigma:[0,1]\to[0,+\infty)$ are measurable functions satisfying
$\rho\in L^{1}(0,1)$ with $\rho(t)\ge\rho_{0}>0$ a.e., and $\sigma\in L^{\infty}(0,1)$.

Setting $u(t)=u(t,\cdot)\in L^{p}(\Omega)$ and
\[
    A=\Delta , \quad \mathscr{F}(t,u(t))(x)=\mathscr{F}_{1}(t,x,u(t,x)),\qquad
    \mathscr{H}[u](x)=\int_{0}^{1}\omega(s)\,u(s,x)\,ds,
\]
 then we can rewrite the system~\eqref{eq:ex1} in the abstract form,
\[
\begin{cases}
    ^CD_t^{\beta}u(t)=Au(t)+\lambda\mathscr{F}(t, u(t))\quad t\in[0,1]\\
    u(0)=\mathscr{H}[u]
\end{cases}
\]
Now, we will verify all the conditions of Theorem~\ref{existence}

For each $v\in\bar{Z}_{\alpha}$, the map $t\mapsto \mathscr{F}(t,u(t))(x)=\rho(t)\phi_{0}(x)
    +\sigma(t)u(t, x)/(1+u(t, x))$ is measurable (product of measurable functions). For each $t\in[0,1]$, the continuity of
    $\mathscr{F}(t,\cdot):\bar{Z}_{\alpha}\to Z$ in the $L^{p}$-norm follows from the
    dominated convergence theorem, since
    $|u(t, x)/(1+u(t, x))|\le u(t, x)$ and $v\mapsto u/(1+u)$ is continuous on $[0,1]$.
    Which show $g$ is a Carath\'eodory map. Moreover, since $u(t, x)/(1+u(t, x))<1$ for all
    $u(t, x)\ge 0$, we have $\mathscr{F}(t,u(t))\in Z$ whenever $u(t)\in Z$, and
    $\|\mathscr{F}(t,v)\|_{L^{p}}\le\|\rho\|_{L^{1}}\|\phi_{0}\|_{L^{p}}
    +\|\sigma\|_{L^{\infty}}|\Omega|^{1/p}$
    uniformly in $(t,u(t))\in[0,1]\times\bar{Z}_{\alpha}$. This show that $\mathscr{F}$ sends bounded subsets of $[0,1]\times\bar{Z}_{\alpha}$ into
    bounded subsets of $Z$. Furthermore, for every $u\in\partial K_{\alpha}$ and every $t\in[0,1]$,
    \[
        \mathscr{F}(t,u(t))(x)
        =\rho(t)\phi_{0}(x)+\sigma(t)\,\frac{u(t,x)}{1+u(t,x)}
        \succcurlyeq \rho(t)\phi_{0}(x)
        =:\gamma_{\alpha}(t)(x).
    \]
    Since $\rho(t)\ge\rho_{0}>0$ a.e.\ and $\phi_{0}\in Z\setminus\{0\}$, the
    function $\gamma_{\alpha}:[0,1]\to Z$ is nontrivial. Hence, hypotheses $(f1)$ and $(f2)$ hold.

The operator $\mathscr{H}:C([0,1];\bar{Z}_{\alpha})\to Z$ defined by
    $\mathscr{H}[u](x)=\int_{0}^{1}\omega(s)u(s,x)\,ds$ is linear and bounded, since
    \[
        \|\mathscr{H}[u]\|_{L^{p}}
        \le\|\omega\|_{L^{1}}\sup_{s\in[0,1]}\|u(s)\|_{L^{p}}
        =\|\omega\|_{L^{1}}\|u\|_{\infty}\le\|\omega\|_{L^{1}}\alpha.
    \]
    For $u\in\partial K_{\alpha}$, for cone lower bound ~\eqref{cone-bound} and the assumption $\omega(s)\ge \omega_{0}>0$ a. e., give 
    \[
        \mathscr{H}[u]= \int_0^1 \omega(s)u(s, x)ds \ge \omega_0 \alpha \sigma e_1(x)=\Phi_{\alpha}(x)
    \]
    Since, $\omega_{0},\sigma,\alpha>0$ and $e_{1}\in Z\setminus\{0\}$ we have $\Phi_{\alpha}\in Z\setminus \{0\}$, and hypothesis $(h1)$ holds.

With $\Phi_{\alpha}=\omega_0 \alpha \sigma e_1\in Z\setminus\{0\}$ and $\gamma_{\alpha}(t)=\rho(t)\phi_{0}$, the condition $(h2)$ requires the existence of $t_{0}\in(0,1]$ such that
    \[
        \left\|\mathcal{S}_{\alpha}\Phi_{\alpha}+\int_{0}^{t_{0}}(t_{0}-r)^{\beta-1}
        \mathcal{T}_{\beta}(t_{0}-r)\,\gamma_{\alpha}(r)\,dr\right\|_{L^{p}}>0.
    \]
    Since $\mathcal{S}_{\alpha}(t_0)$ and $\mathcal{T}_{\beta}(t_{0}-r)$ are positive operators, and
    $\Phi_{\alpha}=\omega_{0}\sigma_{0}\alpha\,e_{1}> 0$. By strict positivity of the heat semigroup generated by $A=\Delta$, strict positivity holds. Taking $t_{0}=1$ yields $(h2)$.

Under the above conditions on $\rho$, $\sigma$, $\omega$, and $\phi_{0}$, by Theorem \ref{existence} there exist a positive eigenvalue $\lambda_{\alpha}\in(0,+\infty)$ and a corresponding nonnegative mild eigenfunction $u_{\alpha}\in\partial K_{\alpha}$ solving the problem~\eqref{eq:ex1}.
\begin{remark}
    A specific example satisfying all the required assumptions is obtained by taking $\Omega=(0,\pi)\subset\mathbb{R}$,
    $p=2$, $\beta=1/2$. Let $\phi_{0}(x)=\sin x$, $\rho(t)=e^{-t}$, $\sigma(t)=\cos^{2}(\pi t)$,
    and $\omega(s)\equiv 1$. One can verify that $\rho\ge e^{-1}>0$, it is easy to check that $\phi_{0}\in Z\setminus\{0\}$, that $\rho$ is bounded below by $e^{-1}>0$ and that all assumptions hold.
\end{remark}


\section*{Data Availability Statement}

Data sharing is not applicable to this article, as no datasets were generated 
or analysed during the current study. All theoretical results, definitions, 
and proofs are contained within the manuscript.
\section*{Disclosure statement}
The authors declare that they have no competing interests.
\section*{Funding}
The authors received no specific funding for this submission.

\end{document}